\documentclass[a4paper,oneside]{amsart}
\usepackage[dvipdfmx]{hyperref}   
\hypersetup{%
 bookmarksnumbered=true,%
 colorlinks=true,%
 setpagesize=false,%
 pdftitle={},%
 pdfauthor={},%
 pdfsubject={},%
 pdfkeywords={}}

\usepackage{amssymb}
\usepackage{amsmath}
\usepackage{amsthm}
\usepackage{amsfonts}
\usepackage{comment}
\usepackage{enumerate}
\excludecomment{note}

\theoremstyle{plain}
\newtheorem{theorem}{Theorem}[section]
\newtheorem{proposition}[theorem]{Proposition}
\newtheorem{lemma}[theorem]{Lemma}
\newtheorem{corollary}[theorem]{Corollary}

\theoremstyle{definition}

\newtheorem{definition}[theorem]{Definition}
\newtheorem{example}[theorem]{Example}

\theoremstyle{remark}
\newtheorem{remark}[theorem]{Remark}

\newcommand{\nat}{\mathbb{N}}

\newcommand{\real}{\mathbb{R}}

\newcommand{\cK}{\mathcal{K}}

\newcommand{\ol}[1]{\overline{#1}}
\newcommand{\ul}[1]{\underline{#1}}

\begin{document}
\title[Derandomization in game-theoretic probability]{Derandomization in game-theoretic probability}

\author{Kenshi Miyabe}
\address[K.~Miyabe]{The University of Tokyo, Japan}
\email{research@kenshi.miyabe.name}

\author{Akimichi Takemura}
\address[K.~Takemura]{The University of Tokyo, Japan}
\email{takemura@stat.t.u-tokyo.ac.jp}

\begin{abstract}
We 
give a general method for constructing  a deterministic strategy
of Reality from a randomized strategy in game-theoretic probability.
The construction can be seen as derandomization in game-theoretic probability.
\end{abstract}


\date{\today}

\maketitle


\section{Introduction}

\subsection{Reality's strategy in Game-theoretic probability}
\label{subsec:intro-reality-strategy}

Game-theoretic probability \cite{ShaVov01}
is a probability theory based on a betting game between two players, Skeptic and Reality.
Sometimes we add the third player called Forecaster.
In game-theoretic probability an almost sure event
is (usually) formalized as an event 
such that Skeptic can increase his capital to infinity
without risking bankruptcy
if the event does not happen.
In this case we say that Skeptic has a winning strategy.
In game-theoretic probability,
in order to prove that an event happens almost surely,
we construct a winning strategy of Skeptic.
A number of such strategies have been constructed so far.
Often these strategies of Skeptic correspond to well-known proofs in measure-theoretic probability that a certain event happens with probability one.  

In this paper, when we just refer to a strategy, 
it is a deterministic strategy.  We explicitly say ``randomized strategy'', when a strategy utilizes 
random variables in the sense of measure-theoretic probability.

There is no counterpart of Reality's strategy in
measure-theoretic probability, because  in measure-theoretic probability Reality is simply generating
random variables under a given probability distribution without any 
specific strategy.  
Hence it is more difficult to derive results on Reality's strategies.
Reality's strategies  correspond to  the notion of derandomization, since if
Reality is following a strategy she is not random in the measure-theoretic sense.


If a two-player game (without Forecaster) is with perfect information
(and the winning set satisfies some regurality condition),
then at least one of the players has a winning strategy in the game
by Martin's Theorem (Theorem \ref{th:martin}).
If Skeptic does not have a winning strategy, Reality should have.
%
For instance, consider the game-theoretic version of
Kolmogorov's strong law of large numbers \cite[Proposition 4.1]{ShaVov01}.
Shafer and Vovk proved the existence
of Reality's winning strategy in Section 4.3 in their book \cite{ShaVov01}.
The proof is, however, nonconstructive.
The two main tools of the proof are
\begin{enumerate}
\item the randomized strategy for Reality that was devised by Kolmogorov,
\item Martin's theorem.
\end{enumerate}
It is unnatural that we need to use such a big theorem, Martin's theorem,
to answer such a simple question.
It was a long-standing question to give a concrete deterministic
strategy of Reality.

Vovk \cite{Vov13} finally gave such a strategy.
The proof is simple, which is a desired property.
The proof seems to be based on the randomized strategy by Kolmogorov,
but it is not clear how these two strategies are related.
Thus it is difficult to know from his proof how to modify the randomized strategy
to answer a similar question in different games.
On the other hand, 
Miyabe and Takemura \cite{Miyabe_convrandseries}
derived a rather strong result on strategies of Reality.
The essential idea is that Reality uses a ``fictional'' strategy of Skeptic.
In fact, Theorem 4.12 of Miyabe and Takemura \cite{Miyabe_convrandseries}
showed the existence of Reality's strategy.
However, they did not give a concrete strategy.
Thus we did not know how the strategy looks like.

In this paper we construct a concrete strategy of Reality
based on the idea above.
The construction goes as follows.
\begin{enumerate}[{[}I{]}]
\item Take a randomized strategy.
\item Construct a strategy of Skeptic that forces the random event.
\item Construct a strategy of Reality using it.
\end{enumerate}
Each step is straightforward and does not require coming up with a new strategy.
Since we construct a deterministic strategy from a randomized one,
we call this \emph{derandomization in game-theoretic probability}.

\subsection{Derandomization}

Randomized algorithm \cite{MotRag95,MitUpf05} has been frequently used
in complexity theory \cite{Gol08,AroBar09}.
One of the reasons is that 
there are some problems such that
it seems difficult to prove that they are polynomial-time computable
but they are polynomial-time computable
with an random oracle with high probability.

The class $\textbf{BPP}$ (Bounded-error Probabilistic Polynomial-time) is, roughly speaking,
the set of problems that are polynomial-time computable
with a randomized algorithm.
It has been conjectured that every problem in $\textbf{BPP}$ is actually
polynomial-time computable,
that is, $\textbf{BPP}=\textbf{P}$.
In other words, the conjecture is asking whether we can always derandomize
in this setting.

An analogous question in computability theory \cite{Coo04,Odi90,Odi99}
has been solved.
On Cantor space $2^\nat$ with the uniform measure $\mu$,
if $A\in2^\nat$ is not computable,
\begin{note}
\marginpar{$2^\omega=2^\nat$. Logician usually use $2^\omega$ and computer scientists use $2^\nat$. For usual mathematician, $2^\nat$ is better.}
\end{note}
then the set of all sequences that compute $A$ has measure $0$
\cite{LMSS56,Sac63}.
Thus, if a sequence is computable by a randomized strategy,
then the sequence should be computable.
In other words, we can always derandomize
if we do not care about computational resource.

Derandomization asks the question how 
we can deterministically construct a sequence random enough.
Construction of such a sequence has been studied
in the theory of algorithmic randomness \cite{Dow10,Nie09}
to separate some randomness notions.
The essential idea is diagonalization.
One recent interesting application is 
the construction of an absolutely normal number
in polynomial time \cite{May,FigNie,BecHeiSla}.

The same technique can be applied to construct a strategy of Reality
that complies with an event in game-theoretic probability.
Derandomization itself is much easier in this case
because we do not care about computability at all.
In contrast, we need to consider a sequence of reals in our case
while derandomization in complexity theory and the theory of algorithmic randomness usually considers an infinite binary sequence.

\subsection{Overview of this paper}

The main theme of this paper is the construction of Reality's strategy.
In Section \ref{sec:BC}
we study Borel-Cantelli lemmas in game-theoretic probability.
This is a simple case and illustrates how the construction goes.
Then, the result will be used in the next section.
In Section \ref{sec:derandomize}
we give a deterministic strategy of Reality that complies with 
the success and the failure of the strong law of large numbers.
In Section \ref{sec:compliance}
we give a general theory of the notion of compliance
and look at some examples.

\section{Preliminaries}\label{sec:pre}


In this paper we mainly consider the unbounded forecasting game defined in Chapter 4 of
Shafer and Vovk \cite{ShaVov01}.

\begin{quote}
{\sc Unbounded Forecasting Game} (UFG)\\
\textbf{Players}: Forecaster, Skeptic, Reality\\
\textbf{Protocol}:\\
\indent $\cK_0:=1$.\\
\indent FOR $n=1,2,\ldots$:\\
\indent\indent Forecaster announces $m_n\in\mathbb{R}$ and $v_n\ge0$.\\
\indent\indent Skeptic announces $M_n\in\mathbb{R}$ and $V_n\ge0$.\\
\indent\indent Reality announces $x_n\in\mathbb{R}$.\\
\indent\indent $\cK_n:=\cK_{n-1}+M_n(x_n-m_n)+V_n((x_n-m_n)^2-v_n)$.\\
\textbf{Collateral Duties}:
Skeptic must keep $\cK_n$ non-negative.
Reality must keep $\cK_n$ from tending to infinity.
\end{quote}

An infinite sequence $\xi=(m_1,v_1,x_1,m_2,v_2,x_2,\cdots)$
of moves of Forecaster and Reality is called a \emph{path}.
Define the sample space
\[\Omega=\{\xi=(m_1,v_1,x_1,m_2,v_2,x_2,\cdots)\ :\ m_n\in\real,\ v_n\ge0,\ x_n\in\real\}\]
as the set of paths.
Any subset $E\subseteq\Omega$ is called an \emph{event}.
We say that a strategy $P$ of Skeptic \emph{forces} an event $E$
if the capital $\cK_n^P(\xi)$ of Skeptic with 
$P$ is non-negative for all $\xi\in\Omega$ and for all $n\ge0$,
and $\xi\not\in E$ implies $\limsup_n\cK_n^P(\xi)=\infty$.
Skeptic \emph{can force} an event if there is a strategy $P$ of Skeptic
that forces the event.  Note that we are not distinguishing ``weak forcing''
and ``forcing'', since they are equivalent (\cite[Lemma 3.1]{ShaVov01}).

\begin{definition}[Miyabe and Takemura \cite{Miyabe_convrandseries}]
By a strategy $R$,
Reality \emph{complies} with an event $E\subseteq\Omega$
if
\begin{enumerate}
\item $\xi\in E$, irrespective of the moves of Forecaster and Skeptic,
with Skeptic observing his collateral duty, 
\item $\sup_n \cK_n<\infty$.
\end{enumerate}
Reality \emph{strongly complies} with $E$ by the strategy $R$
if (ii) is replaced with
$\cK_n\le \cK_0$ for all $n$.
\end{definition}

\begin{theorem}[Shafer and Vovk {\cite[Proposition 4.1]{ShaVov01}}]\label{th:slln}
In the unbounded forecasting game,
\begin{enumerate}
\item Skeptic can force
\[\sum_{n=1}^\infty\frac{v_n}{n^2}<\infty
\Rightarrow\lim_{n\to\infty}\frac{1}{n}\sum_{i=1}^n(x_i-m_i)=0.\]
\item Reality can comply with
\[\sum_{n=1}^\infty\frac{v_n}{n^2}=\infty
\Rightarrow\left(\frac{1}{n}\sum_{i=1}^n(x_i-m_i)\mbox{ does not converge to }0\right).\]
\end{enumerate}
\end{theorem}

We call the event of (i) the \emph{success of SLLN} (Strong Law of Large Numbers)
and the event of (ii) the \emph{failure of SLLN}.
In the proof of (ii),
Shafer and Vovk \cite[Proposition 4.1]{ShaVov01}
use a randomized strategy of Reality and Martin's theorem,
but did not give a concrete strategy.
Vovk \cite{Vov13} gave a concrete strategy.
The result also follows from Theorem 4.12
of Miyabe and Takemura \cite{Miyabe_convrandseries}.

The following is the key fact to give a strategy of Reality.

\begin{theorem}[Miyabe and Takemura {\cite[Proposition 4.10]{Miyabe_convrandseries}}]\label{th:force-comply}
In the unbounded forecasting game,
if Skeptic can force an event $E$,
then Reality can strongly comply with $E$.
\end{theorem}

In the proof of this theorem, a strategy of Reality
was constructed using the strategy of Skeptic
that forces the event.

\section{Borel-Cantelli lemmas}\label{sec:BC}

In this section we focus on game-theoretic versions of
Borel-Cantelli lemmas,
which will play an important role
to give 
the strategy of Reality that complies with the success and the failure of SLLN
in the next section.

\begin{quote}
{\sc Coin-Tossing Game}\\
\textbf{Players}: Forecaster, Skeptic, Reality\\
\textbf{Protocol}:\\
\indent $\cK_0:=1$.\\
\indent FOR $n=1,2,\ldots$:\\
\indent\indent Forecaster announces $p_n\in[0,1]$.\\
\indent\indent Skeptic announces $M_n\in\real$.\\
\indent\indent Reality announces $x_n\in \{0,1\}$.\\
\indent\indent $\cK_n:=\cK_{n-1}+M_n(x_n-p_n)$.\\
\textbf{Collateral Duties}:
Skeptic must keep $\cK_n$ non-negative.
Reality must keep $\cK_n$ from tending to infinity.
\end{quote}

The following result is 
a game-theoretic version of L\'evy's extension of the Borel–Cantelli Lemma.

\begin{proposition}[Miyabe-Takemura {\cite[Example 2.3]{Miyabe_convrandseries}}]
\label{prop:BC}
In the coin-tossing game Skeptic can force
\begin{align}\label{eq:BC}
\sum_n p_n<\infty\iff \sum_n x_n<\infty.
\end{align}
\end{proposition}

By essentially the same argument in the proof of Theorem \ref{th:force-comply},
we can show that Reality can strongly comply with the event \eqref{eq:BC}.
Here we give a concrete strategy.

In Step I, we take a randomized strategy.
In measure theoretic probability,
a counterpart of \eqref{eq:BC} is already shown.
Then, we consider the following derived randomized strategy:
$x_n=1$ with probability $p_n$ and $x_n=0$ with probability $1-p_n$.
Then the property \eqref{eq:BC} is true almost surely
in the sense of measure-theoretic probability.

In Step II, construct a strategy of Skeptic that forces the random event.
Here, the random event is \eqref{eq:BC}.
By the correspondence between measure-theoretic probability
and game-theoretic probability,
this is possible by considering a protocol strong enough.
We need a concrete deterministic strategy of Skeptic
that forces \eqref{eq:BC}, and the simpler the better.

\begin{lemma}\label{lem:BC-infinite}
In the coin-tossing game,
Skeptic can force
\[\sum_n p_n=\infty\Rightarrow\sum_n x_n=\infty.\]
\end{lemma}

\begin{proof}
Let
\begin{equation}
\label{eq:heads-position}
H_n=\{k<n\ :\ x_k=1\},\ T_n=\{k<n\ :\ x_k=0\}
\end{equation}
be the sets of time indices of heads and tails before the round $n$ and
let 
\begin{equation}
\label{eq:bn}
b_n=\#H_n
\end{equation}
denote the number of heads before the round $n$.
Note that 
for any distinct $k,j\in H_n$, we have $k<j\Rightarrow b_k<b_j$. Hence
\begin{equation}
\label{eq:Hn-sum}
\sum_{k\in H_n}2^{-b_k-1} \le 1
\end{equation}
for every $n$.

Consider the following strategy of Skeptic:
\[M_n=-2^{-b_n-1}.\]
We claim that this strategy forces the event.

First, we show that this strategy keeps $\cK_n$ non-negative.
Note that
\begin{align*}
\cK_n
=1-\sum_{k\in H_{n+1}}2^{-b_k-1}(1-p_k)-\sum_{k\in T_{n+1}}2^{-b_k-1}(-p_k)
>1-\sum_{k\in H_{n+1}}2^{-b_k-1}\ge0.
\end{align*}

Next, suppose that $\sum_n p_n=\infty$ and
$\sum_n x_n<\infty$.
Then, there exists $N$ such that
$n\ge N\Rightarrow x_n=0$.
Thus, for every $n\ge N$,
\[\cK_n=\cK_{N-1}+2^{-b_N-1}\sum_{k=N}^n p_k\to\infty.\]
\end{proof}

\begin{lemma}\label{lem:BC-finite}
In the coin-tossing game,
Skeptic can force
\[\sum_n p_n<\infty\Rightarrow\sum_n x_n<\infty.\]
\end{lemma}

\begin{proof}
Let $c_n$ be the natural number satisfying
\begin{equation}
\label{eq:cn}
c_n-1\le\sum_{k=1}^np_k<c_n.
\end{equation}
Note that
\[
\sum_{k : c_k = i} p_k = \sum_{k : c_k\le i} p_k - \sum_{k : c_k\le i-1} p_k < i - (i-2)=2.
\]

Consider the following strategy of Skeptic:
\[M_n=2^{-c_n-1}.\]
We claim that this strategy forces the event.

First we show that this strategy keeps $\cK_n$ non-negative.
Note that
\begin{align}\label{eq:pn}
\sum_{k\in T_{n+1}}2^{-c_k-1}p_k\le
\sum_{k=1}^\infty 2^{-c_k-1}p_k\le
\sum_{i=1}^\infty\sum_{k : c_k=i}2^{-i-1}p_k
\le\sum_{i=1}^\infty2^{-i}\le1,
\end{align}
where $T_n$ is defined in \eqref{eq:heads-position}.
Then 
\[\cK_n\ge1+\sum_{k\in T_{n+1}}M_k(x_k-p_k)>1-\sum_{k\in T_{n+1}}2^{-c_k-1}p_k\ge0.\]

Next we show that this strategy forces the event.
Assume that $\sum_n p_n<\infty$ and $\sum_n x_n=\infty$.
Let $c$ be the natural number such that $c-1<\sum_n p_n\le c$.
Then there exists $N_0$ such that
$c-1<\sum_{k=1}^{N_0} p_k$.
Note that, for $n>N_0$, we have $c\le c_n\le c+1$.
Since $\sum_n p_n<\infty$, there exists $N_1$ such that
$n\ge N_1\Rightarrow p_n<1/2$.
Let $N=\max\{N_0,N_1\}$.
For $n\ge N$ such that $x_n=1$, we have
\[\cK_n-\cK_{n-1}
=2^{-c_n-1}(1-p_n)
\ge 2^{-c-2}\cdot\frac{1}{2}\ge 2^{-c-3}.\]
For $n\ge N$ such that $x_n=0$, we have
\[\cK_n-\cK_{n-1}
=2^{-c_n-1}(0-p_n)
\ge -2^{-c-1}p_n.\]
Since $\sum_n p_n<\infty$, we have $\limsup_n \cK_n=\infty$.
\end{proof}



In Step III, construct a strategy of Reality using it.
First note that Skeptic can force \eqref{eq:BC}
by combining the strategies in Lemma \ref{lem:BC-infinite}
and Lemma \ref{lem:BC-finite}.
Using these strategies,
Reality can strongly comply with the same event
with the technique of Theorem \ref{th:force-comply}.
We give the derived strategy first,
and explain how to derive it in detail later.

\begin{theorem}\label{th:BC-reality}
In the coin-tossing game,
Reality can strongly comply with
\[\sum_n p_n<\infty\iff\sum_n x_n<\infty.\]
\end{theorem}

\begin{proof}
Let $b_n, c_n$ be defined by \eqref{eq:bn} and \eqref{eq:cn}.
We claim that the following strategy of Reality strongly complies with the event:
\begin{quote}
Reality waits for the first round $n_0$ such that
Skeptic announces $M_n\ne0$.
If such a round does not exist, let $n_0=\infty$.

For $n<n_0$ including the case $n_0=\infty$,
Reality announces $x_n$ as
\[x_n=\begin{cases}1&\mbox{ if }c_n\ne c_{n-1}\\
0&\mbox{ if }c_n=c_{n-1}.\end{cases}\]

For $n=n_0$,
Reality announces $x_n$ as
\[x_n=\begin{cases}1&\mbox{ if }M_n<0\\
0&\mbox{ if }M_n>0.\end{cases}\]

If $\cK_{n_0}=0$, then for $n>n_0$ Reality announces $x_n$ as
\[x_n=\begin{cases}1&\mbox{ if }c_n\ne c_{n-1}\\
0&\mbox{ if }c_n=c_{n-1}.\end{cases}\]

If $\cK_{n_0}>0$, then let 
\[\epsilon=1-\frac{\cK_{n_0}}{\cK_0}.\]
For $n>n_0$,
Reality announces $x_n$ as
\[x_n=\begin{cases}1&\mbox{ if }M_n\le d_n\\
0&\mbox{ if }M_n>d_n\end{cases}\]
where
\begin{align}\label{eq:dn}
d_n=\frac{\epsilon\cK_{n_0}}{1-\epsilon}(2^{-b_n-2}-2^{-c_n-2}).
\end{align}
\end{quote}

We show that this strategy strongly complies with the event.
If $n_0=\infty$, then clearly $\sum_n p_n<\infty$ if and only if $\sum_n x_n<\infty$, and $\cK_n=\cK_0$.
Then, we can assume that $n_0<\infty$.

Since $\cK_{n_0}-\cK_{n_0-1}<0$, we have $\cK_{n_0}<\cK_0$.
If $\cK_{n_0}=0$, then Skeptic should announce $M_n=0$ for every $n>n_0$
in order to keep $\cK_n$ non-negative.
Thus, $\cK_n=0$ for every $n>n_0$.
In this case clearly $\sum_n p_n<\infty$ if and only if $\sum_n x_n<\infty$.

In what follows, we assume $\cK_{n_0}>0$. Then $0<\epsilon<1$.
Suppose that $\sum_n p_n=\infty$ and $\sum_n x_n<\infty$.
Since $b_n$ is bounded and $c_n$ goes to infinity,
there exists $\delta>0$ such that $d_n>\delta$ and $x_n=0$ for all sufficiently large $n>N$.
Then for such an $n$, we have
\begin{align*}
\cK_n
=&\ \cK_{N}-\sum_{k=N+1}^n M_kp_k
\le\cK_{N}-\sum_{k=N+1}^n d_k p_k\\
\le&\ \cK_{N}-\sum_{k=N+1}^n \delta p_k\to-\infty
\end{align*}
as $n\to\infty$.
Thus, such a strategy of Skeptic is not allowed.

Suppose that $\sum_n p_n<\infty$
and $\sum_n x_n=\infty$.
Since $c_n$ is bounded, $b_n$ goes to infinity
and $p_n$ goes to $0$,  there exists $\delta>0$ such that
$d_n<-\delta$ and $p_n\le1/2$
for all sufficiently large $n>N$.
For $n> N$ such that $x_n=1$,
\begin{align*}
\cK_n-\cK_{n-1}=M_n(1-p_n)\le d_n(1-p_n)\le-\delta(1-p_n)\le-\frac{\delta}{2}.
\end{align*}
For $n>N$ such that $x_n=0$,
we have
\[M_n>d_n\ge-2^{-c_n-2}\frac{\epsilon\cK_{n_0}}{1-\epsilon},\]
and
\begin{align*}
\cK_n-\cK_{n-1}=-M_np_n\le 2^{-c_n-2}\frac{\epsilon\cK_{n_0}}{1-\epsilon}p_n.
\end{align*}
Thus, $\cK_n\to-\infty$ and such a strategy of Skeptic is not allowed.

Finally we show that $\sup_n \cK_n\le1$.
Since we have $\cK_0=(1-\epsilon)\cK_{n_0}$,
it suffices to show that
\[\cK_n\le\frac{\cK_{n_0}}{1-\epsilon}=\cK_{n_0}+\frac{\epsilon\cK_{n_0}}{1-\epsilon}.\]
For $n\ge n_0$ such that $x_n=1$,
we have
\[\cK_n-\cK_{n-1}
\le M_n(1-p_n)
\le d_n
\le\frac{\epsilon\cK_{n_0}}{1-\epsilon}2^{-b_n-2}.\]
For $n\ge n_0$ such that $x_n=0$,
we have
\[\cK_n-\cK_{n-1}
=-M_np_n
\le\frac{\epsilon\cK_{n_0}}{1-\epsilon}2^{-c_n-2}p_n.\]
By \eqref{eq:Hn-sum} and \eqref{eq:pn}, we have
$\cK_n-\cK_{n_0}\le\dfrac{\epsilon\cK_{n_0}}{1-\epsilon}$.
\end{proof}

From now on we explain how we derived the above strategy.
The goal is to construct a strategy of Reality that complies with the event $E$
of $\sum_n p_n<\infty\iff\sum_n x_n<\infty$.
It suffices to give a strategy with which
Reality's move is ``random'' in the following two senses:
\begin{enumerate}
\item[(a)] The capital is bounded.
\item[(b)] The path satisfies the almost-sure property $E$.
\end{enumerate}
The meaning of randomness in measure-theoretic probability is not clear.
In the theory of algorithmic randomness,
one formulation of randomness is finiteness of the capital
for all betting strategies that are effective in some sense.
In game-theoretic probability,
randomness of a path means the finiteness of the capital in the game.
For instance, Vovk and Shen \cite{VovShe10} have used the terminology of
``game-random''.
With this view, we express the property (b) by the finiteness of the capital
with respect to the strategy with which Skeptic can force the event $E$.

By the proof of Lemma \ref{lem:BC-infinite} and Lemma \ref{lem:BC-finite},
the following strategy $F$ of Skeptic forces the event $E$:
\[M_n=2^{-c_n-2}-2^{-b_n-2}.\]
Reality uses this strategy $F$ as a fictional strategy.
We denote by $S$ the real strategy of Skeptic.
In order to comply with the event $E$,
all Reality has to do is to make the capital with the strategy $(S+F)/2$
finite.
Note that the finiteness of the capital with $(S+F)/2$
implies the finiteness of the capital with $S$ and the capital with $F$.
Furthermore, the finiteness of the capital with $F$
implies the event $E$ because $F$ forces the event $E$.
The strategy $O=(S+F)/2$ announces
\[M_n^O=\frac{M_n^S+2^{-c_n-2}-2^{-b_n-2}}{2}.\]
Reality can make $\cK_n^O\le\cK_{n-1}^O$ by announcing
\[x_n=\begin{cases}1&\mbox{ if }M_n^O\le0\\
0&\mbox{ if }M_n^O>0.\end{cases}\]
Note that 
\[M_n^O\le0\iff M_n^S\le2^{-b_n-2}-2^{-c_n-2}.\]
Then this strategy complies with the event $E$.
Notice that this strategy gives the essential part of
$d_n$ in \eqref{eq:dn}
in the proof of Theorem \ref{th:BC-reality}.

To give a strategy that ``strongly'' complies with the event,
we need an additional little trick.
The idea is taken from the proof of Proposition 4.10 in
\cite{Miyabe_convrandseries}.
A rough idea is as follows.
Wait until the round $n_0$ satisfying $M_{n_0}\ne0$
so that Reality can make the capital strictly less than
the initial capital.
Let $1-\epsilon$ be the ratio of the capital at $n_0$ and the initial capital.
After the round $n_0$,
Reality only has to make the capital with $(1-\epsilon)S+\epsilon F$
non-increasing.
The derived strategy is the strategy in the proof of Theorem \ref{th:BC-reality}.

\section{Derandomization in the Unbounded Forecasting Game and its generalization}\label{sec:derandomize}

In this section we give Reality's strategies complying with the success and the failure of the strong law of large numbers at the same time in the Unbounded Forecasting Game and its generalization.

\subsection{A strategy of Reality for the Unbounded Forecasting Game}


\begin{theorem}\label{th:SLLN-comply}
In the unbounded forecasting game,
Reality can strongly comply with
\[\sum_n \frac{v_n}{n^2}<\infty\iff\lim_{n\to\infty}\frac{1}{n}\sum_{i=1}^n(x_i-m_i)=0.\]
\end{theorem}

Note that this theorem implies (ii) of Theorem \ref{th:slln}.
We show this by giving a concrete strategy of Reality.
In what follows, we assume $m_n=0$ for every $n$ without loss of generality.

In Step I, we take a randomized strategy.
In measure-theoretic probability,
the failure of SLLN was shown by the following randomized strategy
devised by Kolmogorov \cite{Kol30}:
if $v_n<n^2$,
\[x_n:=
\left(\begin{array}{c}
n\\
-n\\
0
\end{array}\right)
\mbox{ with probability }
\left(\begin{array}{c}
v_n/(2n^2)\\
v_n/(2n^2)\\
1-v_n/n^2\\
\end{array}\right)
,\]
respectively;
if $v_n\ge n^2$,
\[x_n:=
\left(\begin{array}{c}
\sqrt{v_n}\\
-\sqrt{v_n}
\end{array}\right)
\mbox{ with probability }
\left(\begin{array}{c}
1/2\\
1/2
\end{array}\right).\]
Here, to show the non-convergence,
the second part of the Borel-Cantelli lemmas is used.
In contrast, if $\sum_n v_n/n^2<\infty$, then
by the first part of the Borel-Cantelli lemma,
SLLN holds almost surely.

In Step II, construct a strategy of Skeptic that forces the random event.
Motivated by the strategy above,
we restrict $x_n\in\{0,\pm n\}$ if $v_n<n^2$,
and $x_n\in\{\pm\sqrt{v_n}\}$ if $v_n\ge n^2$.
Since the property holds almost surely
in the sense of measure-theoretic probability,
it is possible to construct a strategy of Skeptic
that forces the event with this restriction.
With the restriction,
the strategy we need to construct is the one of Skeptic
that forces
\[\sum_n \frac{v_n}{n^2}<\infty\iff x_n=0\mbox{ for all but finitely many }n.\]
Such a strategy of Skeptic can be constructed
by modifying the strategy constructed in the previous section.

In Step III, construct a strategy of Reality using it.
By the technique explained in the previous section,
we can construct a strategy of Reality that strongly complies with the event
with the restriction,
which means that Reality can strongly comply with the event
without the restriction.

In the following proof, we only give the final derived strategy of Reality.

We can forget the round $n$ such that $v_n = 0$ by letting $x_n = 0$.
Thus, we assume that $v_n > 0$ for every $n$.

\begin{proof}
We claim that the following strategy of Reality strongly complies with the event:
\begin{quotation}
Let 
\[b_n=\#\{k<n\ :\ x_k\ne0\}\]
and $c_n$ be the natural number satisfying
\[c_n-1\le\sum_{k=1}^n\frac{v_k}{k^2}<c_n.\]
Reality waits for the first round $n_0$ such that
Skeptic announces $(M_n,V_n)\ne(0,0)$.
If such a round does not exist, let $n_0=\infty$.

For $n<n_0$ including the case $n_0=\infty$,
Reality announces $x_n$ as
\[x_n=\begin{cases}
n&\mbox{ if }c_n\ne c_{n-1}\\
0&\mbox{ if }c_n=c_{n-1}.\end{cases}\]

For $n=n_0$,
Reality announces $x_n$ as
\[x_n=\begin{cases}
1&\mbox{ if }V_n=0\mbox{ and }M_n<0\\
-1&\mbox{ if }V_n=0\mbox{ and }M_n>0\\
0&\mbox{ if }V_n>0.
\end{cases}\]

If $\cK_{n_0}=0$,
then for $n>n_0$ Reality announces $x_n$ as
\[x_n=\begin{cases}
n&\mbox{ if }c_n\ne c_{n-1}\\
0&\mbox{ if }c_n=c_{n-1}.\end{cases}\]

If $\cK_{n_0}>0$,
let 
\[\epsilon=1-\frac{\cK_{n_0}}{\cK_0}.\]
For $n>n_0$ Reality announces $x_n$ as
\[x_n=
\begin{cases}
n&\mbox{ if }v_n<n^2,\ V_n\le d_n\mbox{ and }M_n<0\\
-n&\mbox{ if }v_n<n^2,\ V_n\le d_n\mbox{ and }M_n\ge0\\
0&\mbox{ if }v_n<n^2,\ V_n> d_n\\
\sqrt{v_n}&\mbox{ if }v_n\ge n^2\mbox{ and }M_n<0\\
-\sqrt{v_n}&\mbox{ if }v_n\ge n^2\mbox{ and }M_n\ge0 , 
\end{cases}\]
where
\[d_n=\frac{\epsilon\cK_{n_0}}{1-\epsilon}\frac{2^{-b_n-2}-2^{-c_n-2}}{n^2}.\]
\end{quotation}

We show that this strategy strongly complies with the event.
If $n_0=\infty$, then clearly $\sum_n v_n/n^2<\infty$ if and only if
$x_n=0$ for all but finitely many $n$, and $\cK_n=\cK_0$ for every $n$.
Then, we can assume that $n_0<\infty$.

Consider the round $n=n_0$.
If $V_n=0$, then
\[\cK_{n_0}-\cK_{n_0-1}=-|M_n|<0.\]
If $V_n>0$, then
\[\cK_{n_0}-\cK_{n_0-1}=-V_nv_n<0.\]
Thus, $\cK_{n_0}<\cK_0$.

If $\cK_{n_0}=0$, then Skeptic should announce $(M_n,V_n)=(0,0)$
for every $n>n_0$ in order to keep $\cK_n$ non-negative.
Thus, $\cK_n=0$ for every $n>n_0$.
In this case clearly $\sum_n v_n/n^2<\infty$ if and only if
$x_n=0$ for all but finitely many $n$.

In what follows we assume $\cK_{n_0}>0$.
Then $0<\epsilon<1$.

Suppose that $\sum_n v_n/n^2=\infty$ and 
$x_n=0$ for all but finitely many $n$.
Since $b_n$ is bounded and $c_n$ goes to infinity,
there exists $\delta$ such that $d_n>\delta/n^2$
for all sufficiently large $n>N$.
Then, for such an $n$,
\begin{align*}
\cK_n
= &\ \cK_{N}-\sum_{k=N+1}^n V_kv_k
\le\cK_{N}-\sum_{k=N_0+1}^n d_k v_k\\
\le &\ \cK_{N}-\sum_{k=N+1}^n \delta\cdot\frac{v_k}{k^2}\to-\infty
\end{align*}
as $n\to\infty$.
Thus, such a strategy of Skeptic is not allowed.

Suppose that $\sum_n v_n/n^2<\infty$
and $x_n\ne0$ for infinitely many $n$.
Since $c_n$ is bounded and $b_n$ goes to infinity,
$d_n$ is negative for all sufficiently large $n$.
Since $\sum_n v_n/n^2<\infty$, we have $v_n<n^2$ for all sufficiently large $n$.
Thus, $x_n$ should be $0$ for all sufficiently large $n$.
This is a contradiction.

Finally we show that $\sup_n \cK_n\le1$.
It suffices to show that
$\cK_n\le\cK_{n_0}+\frac{\epsilon\cK_{n_0}}{1-\epsilon}$.
For $n\ge n_0$ such that $v_n<n^2$ and $V_n\le d_n$,
we have
\[\cK_n-\cK_{n-1}\le V_n(x_n^2-v_n)
\le d_n n^2
\le \frac{\epsilon\cK_{n_0}}{1-\epsilon}2^{-b_n-2}.\]
For $n\ge n_0$ such that $v_n<n^2$ and $V_n>d_n$,
we have
\[\cK_n-\cK_{n-1}=-V_nv_n\le0.\]
For $n\ge n_0$ such that $v_n\ge n^2$,
we have
\[\cK_n-\cK_{n-1}\le V_n(x_n^2-v_n)=0.\]
Thus,
$\cK_n-\cK_{n_0}\le\dfrac{\epsilon\cK_{n_0}}{1-\epsilon}$.
\end{proof}

\subsection{A strategy of Reality in a generalization of the Unbounded Forecasting Game}

There are some possible ways in which we generalize our result for the Unbounded Forecasting Game.
Kumon, Takemura and Takeuchi \cite{KumTakTak07}
have obtained similar results in a game which generalizes the Unbounded Forecasting Game.
Miyabe and Takemura \cite{Miyabe_convrandseries}
have shown the existence of the strategy 
that strongly complies with the failure of SLLN
in a rather general setting.
Here,
we give a stronger result in the general setting
in the following senses.
\begin{enumerate}
\item We give a concrete deterministic strategy.
\item The strategy strongly complies with the success and the failure of SLLN
at the same time.
\item We use weaker assumptions.
\item The strategy is much simpler.
\end{enumerate}


The following protocol is from Section 5 in Miyabe and Takemura \cite{Miyabe_convrandseries}.

\begin{quote}
{\sc Unbounded Forecasting Game with General Hedge} (UFGH)\\
\textbf{Parameters}: A single function $h:\real\to\real$\\
\textbf{Players}: Forecaster, Skeptic, Reality\\
\textbf{Protocol}:\\
\indent $\cK_0:=1$.\\
\indent FOR $n=1,2,\ldots$:\\
\indent\indent Forecaster announces $m_n\in\mathbb{R}$ and $v_n\ge0$.\\
\indent\indent Skeptic announces $M_n\in\mathbb{R}$ and $V_n\ge0$.\\
\indent\indent Reality announces $x_n\in\mathbb{R}$.\\
\indent\indent $\cK_n:=\cK_{n-1}+M_n(x_n-m_n)+V_n(h(x_n-m_n)-v_n)$.\\
\textbf{Collateral Duties}:
Skeptic must keep $\cK_n$ non-negative.
Reality must keep $\cK_n$ from tending to infinity.
\end{quote}

\noindent
\textbf{Assumption}
\begin{enumerate}
\item[(A0)] $h(x)=h(|x|)\ge0$.
\item[(A1)] $h(x)/x$ is monotone increasing for $x>0$.
\item[(A2)] $h(x)/x^2$ is monotone decreasing for $x>0$.
\item[(A3)] $h(x)=x^2$ for $|x|\le1$.
\end{enumerate}

Here, we are taking into account Remark 5.3 of \cite{Miyabe_convrandseries}.

\begin{theorem}[Theorem 5.9 in \cite{Miyabe_convrandseries}]
Suppose that $h$ satisfies \textrm{(A0)-(A3)} and
that $g$ is a positive increasing function.
Then in UFGH, Reality can strongly comply with
\[\sum_n\frac{v_n}{g(A_n)}=\infty\Rightarrow\sum_{k=1}^n\frac{x_k-m_k}{h^{-1}\circ g(A_k)}\mbox{ does not converge.}\]
\end{theorem}

\begin{theorem}[Theorem 5.10 in \cite{Miyabe_convrandseries}]
Let $h(x)=x^r$ where $1\le r\le 2$
and $g$ be a positive increasing function.
Then in UFGH, Reality can strongly comply with
\[\sum_n\frac{v_n}{g(A_n)}=\infty\Rightarrow\frac{\sum_{k=1}^n(x_k-m_k)}{h^{-1}\circ g(A_n)}\mbox{ does not converge.}\]
\end{theorem}

We consider a slightly weaker condition
by replacing (A3) by (A4) below,
but (A4) is cricital to show the strong compliance.
See Remark \ref{rem:a4} for details.

\bigskip
\noindent
\textbf{Assumption}
\begin{enumerate}
\item[(A4)] $h(0)=0$.
\end{enumerate}

From now on we show the following theorem.

\begin{theorem}\label{th:strong-compliance}
Suppose that $h$ satisfies \textrm{(A0)-(A2), (A4)} and
that $g$ is a positive increasing function.
Then in UFGH, Reality can strongly comply with
\begin{align}\label{eq:comply}
\sum_n\frac{v_n}{g(A_n)}<\infty\iff
\frac{\sum_{k=1}^n(x_k-m_k)}{h^{-1}\circ g(A_n)}\mbox{ converges.}
\end{align}
\end{theorem}

Note that Theorem \ref{th:strong-compliance}
implies the two theorems above.
Theorem \ref{th:strong-compliance} also implies
Proposition 2.1 in \cite{KumTakTak07}
by letting $m_n=0$, $v_n=v$, $h(x)=x^r$ and $g(x)=x/v$.
Furthermore, Theorem \ref{th:strong-compliance} also implies
Proposition 3.1 in \cite{KumTakTak07}
by letting $m_n=0$, $v_n=v$ and $g(x)=h(vx)$.

In what follows, without loss of generality,
we assume that $m_n=0$ for every $n$.
Furthermore, we can forget the round $n$ such that $v_n=0$
by letting $x_n=0$.
Thus, we assume that $v_n>0$ for every $n$.

Before giving the strategy, we recall the following lemma.

\begin{lemma}[Miyabe and Takemura {\cite[Lemma 4.15]{Miyabe_convrandseries}}]
\label{lem:epsilon}
Let $\{a_n\}$ be a sequence of positive reals.
Then there exists a sequence $\{\epsilon_n\}$ of positive reals such that
\begin{enumerate}
\item $\epsilon_n$ is determined only by $a_1,\cdots,a_n$,
\item $\epsilon_n a_n\le1$,
\item $\sum_n a_n=\infty$ implies $\sum_n \epsilon_n a_n=\infty$ and
$\epsilon_n\to0$.
\end{enumerate}
\end{lemma}

Furthermore, by the proof, we can assume that
\begin{enumerate}
\item[(iv)] $\sum_n a_n<\infty$ implies that $\{\epsilon_n\}$ converges to a positive real.
\end{enumerate}

Now we are ready to give the strategy.

In UFGH, we consider the following strategy of Reality:
\begin{quotation}
Let 
\[b_n=\#\{k<n\ :\ x_k\ne0\}\]
and $c_n$ be the natural number satisfying
\[c_n-1\le\sum_{k=1}^n\frac{\epsilon_k v_k}{g(A_k)}<c_n\]
where $\{\epsilon_n\}$ is the sequence determined
by Lemma \ref{lem:epsilon} for $\{v_n/g(A_n)\}$.
Reality waits for the first round $n_0$ such that
Skeptic announces $(M_n,V_n)\ne(0,0)$.
If such a round does not exist, $n_0=\infty$.

For $n<n_0$ including the case $n_0=\infty$,
Reality announces $x_n$ as
\[x_n=\begin{cases}
e_n&\mbox{ if }c_n\ne c_{n-1}\\
0&\mbox{ if }c_n=c_{n-1}.\end{cases}\]
where
\[e_n=h^{-1}(g(A_n)\cdot \epsilon_n^{-1}).\]

For $n=n_0$,
Reality announces $x_n$ as
\[x_n=\begin{cases}
1&\mbox{ if }V_n=0\mbox{ and }M_n<0\\
-1&\mbox{ if }V_n=0\mbox{ and }M_n>0\\
0&\mbox{ if }V_n>0.
\end{cases}\]

If $\cK_{n_0}=0$,
then for $n>n_0$ Reality announces $x_n$ as
\[x_n=\begin{cases}
e_n&\mbox{ if }c_n\ne c_{n-1}\\
0&\mbox{ if }c_n=c_{n-1}.\end{cases}\]

If $\cK_{n_0}>0$,
let 
\[\epsilon=1-\frac{\cK_{n_0}}{\cK_0}.\]
For $n>n_0$ Reality announces $x_n$ as
\[x_n=
\begin{cases}
e_n&\mbox{ if }\epsilon_nv_n<g(A_n),\ V_n\le d_n\mbox{ and }M_n<0\\
-e_n&\mbox{ if }\epsilon_nv_n<g(A_n),\ V_n\le d_n\mbox{ and }M_n\ge0\\
0&\mbox{ if }\epsilon_nv_n<g(A_n),\ V_n> d_n\\
h^{-1}(v_n)&\mbox{ if }\epsilon_nv_n\ge g(A_n)\mbox{ and }M_n<0\\
-h^{-1}(v_n)&\mbox{ if }\epsilon_nv_n\ge g(A_n)\mbox{ and }M_n\ge0
\end{cases}\]
where
\[d_n=\frac{\epsilon\cK_{n_0}}{1-\epsilon}\frac{2^{-b_n-2}-2^{-c_n-2}}{g(A_n)\cdot\epsilon_n^{-1}}.\]
\end{quotation}

We show that this strategy strongly complies
with the success and the failure of SLLN at the same time.
In the proof we use the following lemma.

\begin{lemma}[Miyabe and Takemura {\cite[Lemma 4.14]{Miyabe_convrandseries}}]
\label{lem:term-bound}
Let $\{y_n\}$ be a sequence of reals and
let $\{g_n\}$ be a non-decreasing sequence of positive reals.
If $(\sum_{k\le n}y_k)/g_n$ converges to $d$,
then $|y_n/g_n|\le|d|+1$ for all but finitely many $n$.
\end{lemma}

\begin{proof}[Proof of Theorem \ref{th:strong-compliance}]
By a similar argument as the proof of Theorem \ref{th:SLLN-comply},
we can show that Reality can strongly comply with
\begin{align*}
\sum_n\frac{v_n}{g(A_n)}<\infty
\iff \sum_n\frac{\epsilon_n v_n}{g(A_n)}<\infty
\iff x_n=0\mbox{ for all but finitely many }n
\end{align*}
except the following two places.

(1) The capital of the case such that $\sum_n v_n/g(A_n)=\infty$
and $x_n=0$ for all but finitely many $n$ is as follows:
\begin{align*}
\cK_n
=&\cK_{N}-\sum_{k=N+1}^n V_kv_k
\le\cK_{N}-\sum_{k=N+1}^n d_kv_k\\
\le&\cK_{N}-\sum_{k=N+1}^n \delta\cdot\frac{\epsilon_k v_k}{g(A_k)}\to-\infty.
\end{align*}
(2) For $n\ge n_0$ such that $x_n\ne0$, we have
\begin{align*}
\cK_n-\cK_{n-1}
\le& V_n(h(e_n)-v_n)
\le d_n\cdot h(e_n)\\
\le&\frac{\epsilon \cK_{n_0}}{1-\epsilon}\cdot\frac{2^{-b_n-2}}{g(A_n)\cdot\epsilon_n^{-1}}\cdot g(A_n)\cdot\epsilon_n^{-1}
=\frac{\epsilon \cK_{n_0}}{1-\epsilon}\cdot2^{-b_n-2}
\end{align*}
if $\epsilon_n v_n<g(A_n)$, and
$\cK_n-\cK_{n-1}\le0$ if $\epsilon_n v_n\ge g(A_n)$. 

From now on, we show that
$x_n=0$ for all but finitely many $n$
if and only if
$\sum_{k=1}^n x_k/h^{-1}\circ g(A_n)$ converges.

Suppose that $x_n=0$ for all but finitely many $n$
and $\lim_n A_n<\infty$.
Then, $\sum_n x_n$ converges and $h^{-1}\circ g(A_n)$ converges.
Thus, $\sum_{k=1}^n x_k/h^{-1}\circ g(A_n)$ converges.

Suppose that $x_n=0$ for all but finitely many $n$
and $\lim_n A_n=\infty$.
Then, $\sum_n x_n$ converges and $\lim_n h^{-1}\circ g(A_n)=\infty$.
Thus, $\sum_{k=1}^n x_k/h^{-1}\circ g(A_n)$ converges to $0$.

Suppose that $x_n\ne0$ for infinitely many $n$.
This means that
\[|x_n|=\begin{cases}e_n&\mbox{ if }\epsilon_n v_n<g(A_n)\\
h^{-1}(v_n)&\mbox{ if }\epsilon_n v_n\ge g(A_n)\end{cases}\]
for infinitely many $n$.
Note that, if $\epsilon_n v_n\ge g(A_n)$, then
$h^{-1}(v_n)\ge h^{-1}(g(A_n)\cdot\epsilon_n^{-1})$
by the monotonicity of $h$.
Thus,
\[\frac{|x_n|}{h^{-1}\circ g(A_n)}
\ge\frac{h^{-1}(g(A_n)\cdot \epsilon_n^{-1})}{h^{-1}\circ g(A_n)}\]
for infinitely many $n$.
We claim that the right-hand side goes to infinity.
Then, by Lemma \ref{lem:term-bound},
$\sum_{k=1}^nx_k/h^{-1}\circ g(A_n)$ does not converge.

Since $\epsilon_n\to0$ as $n\to\infty$, we have
\[g(A_n)\le g(A_n)\cdot\epsilon_n^{-1}\]
for all sufficiently large $n$.
Since $h$ is non-decreasing, so is $h^{-1}$ and
\[h^{-1}(g(A_n))\le h^{-1}(g(A_n)\cdot\epsilon_n^{-1}).\]
By Assumption (A2), we have
\[\frac{g(A_n)}{(h^{-1}(g(A_n)))^2}\ge
\frac{g(A_n)\cdot\epsilon_n^{-1}}{(h^{-1}(g(A_n)\cdot\epsilon_n^{-1}))^2},\]
which implies that
\[\frac{h^{-1}(g(A_n)\cdot \epsilon_n^{-1})}{h^{-1}\circ g(A_n)}
\ge\epsilon_n^{-1/2}\to\infty.\]
\end{proof}

\begin{remark}\label{rem:a4}
Notice that Assumption \textrm{(A4)} is used to show that $\cK_{n_0}<\cK_0$.
At the round $n=n_0$, if $V_n>0$, then $x_n=0$ and
\[\cK_{n_0}-\cK_{n_0-1}=V_n(h(0)-v_n),\]
which is negative because $h(0)=0$.
\end{remark}

\section{On the notion of Compliance}\label{sec:compliance}
\begin{note}
\marginpar{title? How about this?}
\end{note}

In this section we give a general theory on compliance.
We consider the unbounded forecasting game or the coin-tossing game,
but most theorems can be applied to a similar game.

\subsection{The strength of compliance}
\begin{note}
\marginpar{title? How about this?}
\end{note}

Recall Theorem \ref{th:force-comply}, which says that
Reality can strongly comply with the event that Skeptic can force.
An interpretation of this fact in the usual notion of probability is like this:
For each event with probability $1$, one can deterministically take
a path in the event.
Clearly, the probability of the event is closely related
to the supremum of the capital of Skeptic.
The strongness of the compliance seems to be due to probability $1$ of the event.
With this motivation we study the relation between the notion of compliance
and the upper and lower probability.



We denote strategies of Forecaster, Skeptic and Reality by $F$, $S$ and $R$ respectively.

\begin{definition}[Shafer and Vovk {\cite[Chapter 8.3]{ShaVov01}}]
The upper probability of an event $E$ is defined as
\[\overline{P}(E)=\inf\{a\ |\ (\exists S)(\forall F)(\forall R)\mathcal{K}_0=a\ \&\ 
E\Rightarrow\sup_n \mathcal{K}_n\ge1\},\]
where $S$ needs to keep the capital non-negative.
The lower probability is defined as
\[\underline{P}(E)=1-\overline{P}(E^c).\]
\end{definition}

The following are some properties of $\overline{P}$ and $\underline{P}$.

\begin{proposition}[see {\cite[Proposition 8.12]{ShaVov01}}]\label{pro:force-P}
Skeptic can force an event $E$ if and only if $\underline{P}(E)=1$.
\end{proposition}

The ``only if'' direction holds because 
the convex combination of strategies  of Skeptic is possible in the game.

\begin{proposition}[see {\cite[Proposition 8.10]{ShaVov01}}]\label{pro:measure}
The upper probability $\overline{P}$ is an outer measure
and the lower probability $\underline{P}$ is an inner measure.
\end{proposition}

\begin{proposition}[see {\cite[Lemma 1]{TakVovSha11}}]\label{pro:P-larger}
For every event $E$, we have
\[0\le\underline{P}(E)\le\overline{P}(E)\le1.\]
\end{proposition}



We define a similar function based on the notion of compliance.

\begin{definition}
For an event $E$, let
\[\overline{Q}(E)=
\sup\{a\ |\ (\exists R)(\forall F)(\forall S)
\mathcal{K}_0=a\Rightarrow E\wedge\sup_n \mathcal{K}_n\le1\}.\]
Let
\[\underline{Q}(E)=1-\overline{Q}(E^c).\]
\end{definition}

The following are immediate by definition.

\begin{proposition}
If Reality can strongly comply with $E$, then $\overline{Q}(E)=1$.
\end{proposition}

\begin{proposition}
If $\overline{Q}(E)>0$, then Reality can comply with $E$.
\end{proposition}



A two-player game is called \emph{determined}
if one of the player has a winning strategy.

\begin{theorem}[Martin's theorem; see {\cite[Chapter 4.6]{ShaVov01}}]
\label{th:martin}
If the winning condition is Borel, then the game is determined.
\end{theorem}

\begin{remark}
In fact, Martin's theorem says that quasi-Borel is enough.
\end{remark}


If Forecaster is not a player in the game,
then the game is with perfect information.
Then, by Martin's theorem,
we can show that $Q$ is equal to $P$ in this case.

\begin{theorem}
Consider that Forecaster's strategy is fixed in advance and Borel.
If $E$ is Borel, then $\overline{P}(E)=\overline{Q}(E)$
and $\underline{P}(E)=\underline{Q}(E)$.
\end{theorem}

\begin{proof}
Suppose the game whose winning condition $C_a$ of Skeptic is
\[\mathcal{K}_0=a\wedge E\Rightarrow\sup_n \mathcal{K}_n\ge1.\]
This condition can be written as
\[A\wedge\left(E^c\vee\bigwedge_m\bigvee_n B_{n,m}\right)\]
where $A$ is $\cK_0=a$ and $B_{n,m}$ is $\cK_n\ge1-2^{-m}$.
Note that $A$ and $B_{n,m}$ are Borel.
Recall that the class of Borel sets is closed under countable union,
countable intersection and complement.
Thus, the condition is Borel for each $a$.
\begin{note}
\marginpar{reference? clear now?}
\end{note}

Consider $\overline{P}(E)=x>0$, and let $\epsilon$ be a positive real small enough.
If $\mathcal{K}_0=x-\epsilon$, then no strategy $S$ can guarantee to win.
Then Reality has a winning strategy by Martin's theorem.
It follows that $\overline{Q}(E)\ge x=\overline{P}(E)$.

By a similar argument,
by interchansing the roles of the players,
we can show $\ol{Q}(E)\le\ol{P}(E)$.
Hence $\ol{Q}(E)=\ol{P}(E)$.

Consider $\overline{P}(E)=0$.
Then for each $x>0$ there exists a strategy $S$ such that
$\cK_0<x/2$ and $E\Rightarrow\sup_n \cK_n\ge 1$.
Thus, the strategy $2S$ forces
$\cK_0<x$ and $E\Rightarrow\sup_n \cK_n>1$.
This means that no Reality's strategy complies with
$\cK_0<x$ and $E\Rightarrow\sup_n \cK_n\le1$.
Hence, $\overline{Q}(E)\le x$.
Since $x$ is arbitrary, we have $\overline{Q}(E)=0$.

Since $E$ is Borel, then so is $E^c$.
Note that $\overline{P}(E^c)=\overline{Q}(E^c)$.
Then $\underline{P}(E)=\underline{Q}(E)$.
\end{proof}


As we will see in Corollary \ref{cor:q-non-linear},
we do not have $\overline{Q}(E)\le\overline{Q}(E)$
or $\underline{Q}(E)\ge\overline{Q}(E)$ in general.
Thus, we introduce the convex closure of them as follows:
\[\mathrm{conv}[\ul{Q}(E),\ol{Q}(E)]
=\begin{cases}[\ul{Q}(E),\ol{Q}(E)]&\mbox{ if }\quad\ul{Q}(E)\le\ol{Q}(E)\\
[\ol{Q}(E),\ul{Q}(E)]&\mbox{ if }\quad\ul{Q}(E)>\ol{Q}(E).\end{cases}\]

Theorem \ref{th:force-comply} says that, in our terminology,
\[\underline{P}(E)=1\Rightarrow \overline{Q}(E)=1.\]
By a little modification, we can show the following.

\begin{proposition}\label{pro:comply-not-full}
Let $E$ be an event.
\begin{enumerate}
\item The intervals $\mathrm{conv}[\ul{Q}(E),\ol{Q}(E)]$
and $[\ul{P}(E),\ol{P}(E)]$ always overlap.
\item $\ol{Q}(E)<\ul{Q}(E)$ is only possible when they are nested:
\[\mathrm{conv}[\ul{Q}(E),\ol{Q}(E)]\subseteq[\ul{P}(E),\ol{P}(E)].\]
\end{enumerate}
\end{proposition}

\begin{proof}
It suffices to show that
\[\underline{P}(E) \le \overline{Q}(E)\mbox{ and }
\underline{Q}(E) \le \overline{P}(E).\]
We can assume $\underline{P}(E)=x>0$.
Then $\overline{P}(E^c)=1-x$.
Hence there exists $S$ such that
$\mathcal{K}_0^S=1-x+\epsilon$,
and $E^c\Rightarrow\sup_n \mathcal{K}_n^S\ge1$
for small enough $\epsilon>0$.

Now consider a strategy $T$ of Skeptic
such that $\mathcal{K}_0^T=x-2\epsilon$.
Then $S+T$ is a strategy such that $\mathcal{K}_0^{S+T}=1-\epsilon$.
Hence Reality has the strategy satisfying
$\sup_n\mathcal{K}_n^{S+T}\le1-\epsilon$.
It follows that $\sup_n \mathcal{K}_n^S\le1-\epsilon$ and $\sup_n \mathcal{K}_n^T\le1-\epsilon$.
Since $E^c\Rightarrow\sup_n\cK_n^S\ge1$, we have $E$.
Then this strategy is the witness of $\overline{Q}(E)\ge x$.
Hence $\ul{P}(E)\le\ol{Q}(E)$.
Considering $E^c$, we also have $\ul{Q}(E)\le\ol{P}(E)$.
\end{proof}

\begin{corollary}
If $\underline{P}(E)>0$, then Reality can comply with $E$.
\end{corollary}

\subsection{Examples}

\subsubsection{Coin-tossing game}

In the following we look at some examples.
Some examples show the difference between $P$ and $Q$.

\begin{example}\label{ex:coin-tossing-finite}
In the coin-tossing game, let $E$ be the event that
$S_n=\sum_{k=1}^n x_k$ is finite.
Then,
\[\ul{P}(E)=0,\quad \ol{P}(E)=1,\quad
\ul{Q}(E)=1,\quad\ol{Q}(E)=0.\]
Skeptic can not force $E$ or $E^c$.
Reality can not comply with $E$ or $E^c$.
\end{example}

Note that $E$ is a tail event.

\begin{proof}
First, we claim that $\ul{P}(E)=0$.
Assume that $\cK_0<1$.
Consider the case that $\cK_0<\prod_{n=1}^\infty p_n<1$
and $x_n=1$ for every $n$.
Since Skeptic needs to keep $\cK_n\ge0$, we have
$\cK_{n-1}+M_n(0-p_n)\ge0$,
thus $M_n\le\frac{\cK_{n-1}}{p_n}$.
Then,
\[\cK_n=\cK_{n-1}+M_n(1-p_n)\le\cK_{n-1}+\frac{1-p_n}{p_n}\cK_{n-1}
=\frac{\cK_{n-1}}{p_n}\le\frac{\cK_0}{\prod_{i=1}^n p_i}.\]
Thus, $\sup_n \cK_n\le\frac{\cK_0}{\prod_{n=1}^\infty p_n}<1$.
Hence $\ol{P}(E^c)=1$ and $\ul{P}(E)=0$.

We claim that $\ol{P}(E)=1$.
Assume that $\cK_0<1$.
Consider the case that $\cK_0<\prod_{n=1}^\infty(1-p_n)<1$
and $x_n=0$ for every $n$.
By an argument similar to above,
$M_n$ must satisfy $-M_n\le\frac{\cK_{n-1}}{1-p_n}$
and $\sup_n\cK_n\le\frac{\cK_0}{\prod_{n=1}^\infty(1-p_n)}<1$.
Hence $\ol{P}(E)=0$.

Since $\ul{P}(E)=0$ and $\ul{P}(E^c)=0$,
Skeptic can not force $E$ or $E^c$.

We claim that $\ul{Q}(E)=1$.
Consider the case that $\sum_n p_n<\infty$
and Skeptic uses a strategy that forces
$\sum_n p_n<\infty\Rightarrow S_n<\infty$.
Then $\sup_n\cK_n=\infty$ if $\cK_0>0$.
This means that $\ol{Q}(E^c)=0$ and $\ul{Q}(E)=1$.

Similarly, we have $\ol{Q}(E)=0$
because Skeptic can force $\sum_n p_n=\infty\Rightarrow S_n=\infty$.

Consider the case that $\sum_n p_n=\infty$
and Skeptic uses a strategy that forces
$\sum_n p_n=\infty\Rightarrow \lim_n S_n=\infty$.
In order to make $\sup_n \cK_n<\infty$,
Reality needs to announce $\{x_n\}$ so that $\lim_n S_n=\infty$.
This means that Reality can not comply with $E$.
A similar argument can be applied for $E^c$.
\end{proof}

\begin{corollary}
The function $\overline{Q}$ is not an outer measure in general
and $\underline{Q}$ is not an inner measure in general.
\end{corollary}

\begin{proof}
In Example \ref{ex:coin-tossing-finite},
we have $\overline{Q}(E)=\overline{Q}(E^c)=0$
and $\ol{Q}(E\cup E^c)=1$.
Thus, the inequality $\ol{Q}(A\cup B)\le\ol{Q}(A)+\ol{Q}(B)$
does not hold even if $A$ and $B$ are disjoint.
\end{proof}

\begin{example}\label{ex:positive-p_1}
In coin-tossing game,
let $E$ be the event that
\[p_1>0\Rightarrow x_1=1.\]
Then,
\[\ul{P}(E)=0,\quad\ol{P}(E)=1,\quad\ul{Q}(E)=1,\quad\ol{Q}(E)=0.\]
Skeptic can not force $E$ or $E^c$.
Reality can comply with $E$ but can not comply with $E^c$.
\end{example}

\begin{proof}
We claim that $\ul{P}(E)=1$.
Assume $\cK_0<1$.
Consider the case that $p_1=x_1=1$.
Then, $\cK_1=\cK_0<1$.
Thus, no strategy of Skeptic gurantees $\sup_n\cK_n\ge1$ even if $E^c$ happens.
Hence, $\ol{P}(E^c)=1$ and $\ul{P}(E)=0$.

We claim that $\ol{P}(E)=1$.
Assume $\cK_0<1$.
Consider the case that $p_1=x_1=0$.
Then, $\cK_1=\cK_0<1$.
Thus, no strategy of Skeptic gurantees $\sup_n\cK_n\ge1$ even if $E$ happens.
Hence, $\ol{P}(E)=1$.

Since $\ul{P}(E)=0$ and $\ul{P}(E^c)=0$,
Skeptic can not force $E$ or $E^c$.

We claim that $\ul{Q}(E)=1$.
Consider the case that $p_1=0$.
Then, $E^c$ can not happen.
Hence, Reality can not comply with $E^c$, $\ol{Q}(E^c)=0$ and $\ul{Q}(E)=1$.

We claim that $\ol{Q}(E)=0$.
Assume $\cK_0>0$.
Consider the case that $p_1=\cK_0/2$ and $M_1=2$.
Then, the capital $\cK_1$ for $x_1=0$ is
\[\cK_1=\cK_0+2(0-p_1)=0\]
and the capital $\cK_1$ for $x_1=1$ is
\[\cK_1=\cK_0+2(1-p_1)=2.\]
Thus, $\ol{Q}(E)=0$.

We claim that Reality can comply with $E$.
Consider the following strategy of Reality:
\[x_1=\begin{cases}0&\mbox{ if }p_1=0\\
1&\mbox{ if }p_1>0\end{cases},\quad
x_n=\begin{cases}0&\mbox{ if }M_n\ge0\\
1&\mbox{ if }M_n<0.\end{cases}\]
Then, $\sup_n\cK_n=\cK_1<\infty$.
Notice that this strategy gurantees that $E$ happens.
Hence, Reality can comply with $E$.
\end{proof}

Recall that, if $\ol{Q}(E)$ is positive,
then Reality can comply with $E$.
Example \ref{ex:positive-p_1} shows that the converse does not hold.

\subsubsection{Bounded forecasting game}

\begin{quote}
{\sc Bounded Forecasting Game}\\
\textbf{Players}: Forecaster, Skeptic, Reality\\
\textbf{Protocol}:\\
\indent $\cK_0:=1$.\\
\indent FOR $n=1,2,\ldots$:\\
\indent\indent Forecaster announces $p_n\in[0,1]$.\\
\indent\indent Skeptic announces $M_n\in\mathbb{R}$.\\
\indent\indent Reality announces $x_n\in[0,1]$.\\
\indent\indent $\cK_n:=\cK_{n-1}+M_n(x_n-p_n)$.\\
\textbf{Collateral Duties}:
Skeptic must keep $\cK_n$ non-negative.
Reality must keep $\cK_n$ from tending to infinity.
\end{quote}

\begin{example}\label{ex:x_n=p_n}
In the bounded forecasting game,
let $E$ be the event that $x_n=p_n$ for infinitely many $n$.
Then,
\[\ul{P}(E)=0,\quad\ol{P}(E)=1,\quad\ul{Q}(E)=0,\quad\ol{Q}(E)=1.\]
Skeptic can not force $E$ or $E^c$.
Reality can comply with $E$ and $E^c$.
In contrast, Reality can not strongly comply with $E^c$.
\end{example}

\begin{proof}
We claim that $\ul{P}(E)=0$.
Assume $\cK_0<1$.
Consider the case that $p_n=1/2$ for every $n$
and
\[x_n=\begin{cases}1&\mbox{ if }M_n\le0\\
0&\mbox{ if }M_n>0.\end{cases}\]
Then, $E^c$ happens but $\sup_n\cK_n\le\cK_0<1$.
Hence, $\ol{P}(E^c)=1$ and $\ul{P}(E)=0$.

We claim that $\ol{P}(E)=1$.
Assume $\cK_0<1$.
Consider the case that $p_n=x_n=1/2$ for every $n$.
Then, $E$ happns but $\sup_n\cK_n\le\cK_0<1$.
Hence, $\ol{P}(E)=1$.

Since $\ul{P}(E)=\ul{P}(E^c)=0$,
Skeptic can not force $E$ or $E^c$.

We claim that $\ul{Q}(E)=0$.
Assume $\cK_0<1$.
Let $q$ be a rational such that $\cK_0<q<1$.
Consider the following strategy of Reality: for every $n$,
\begin{enumerate}
\item if $p_n=0$, then take $x_n>0$ so that
$M_nx_n\le\frac{q-\cK_{n-1}}{2}$,
\item if $p_n=1$, then take $x_n<1$ so that
$M_n(x_n-1)\le\frac{q-\cK_{n-1}}{2}$,
\item if $p_n\in(0,1)$ and $M_n\le0$, then $x_n=1$,
\item if $p_n\in(0,1)$ and $M_n>0$, then $x_n=0$.
\end{enumerate}
Then, inductively, we have 
\[\cK_n
=\cK_{n-1}+M_n(x_n-p_n)
\le\cK_{n-1}+\frac{q-\cK_{n-1}}{2}=\frac{q+\cK_{n-1}}{2}.\]
Hence, $\sup_n\cK_n\le q<1$.
This means that $\ol{Q}(E^c)=1$ and $\ul{Q}(E)=0$.
By $\ol{Q}(E^c)>0$, Reality can comply with $E$.

We claim that $\ol{Q}(E)=1$.
This is verified by considering the strategy of Reality
that just announces $x_n=p_n$ for every $n$.
By $\ol{Q}(E)>0$, Reality can comply with $E$.

Finally, we claim that Reality can not strongly comply with $E^c$.
Consider the case that $p_n=0$ and $M_n=1$ for every $n$.
Notice that Skeptic can keep $\cK_n\ge0$ by this strategy.
To keep $\cK_n\le\cK_0$ for every $n$,
Reality needs to announce $x_n=0$ for every $n$.
This means that Reality can not strongly comply with $E^c$.
\end{proof}

\begin{corollary}\label{cor:q-non-linear}
The function $\overline{Q}$ is not an inner measure in general
and $\underline{Q}$ is not an outer measure in general.
\end{corollary}

\begin{proof}
Notice that $\ol{Q}(E)=\ol{Q}(E^c)=\ol{Q}(E\cup E^c)=1$.
\end{proof}

\begin{corollary}\label{cor:not-order}
We do not have $\ul{Q}(E)\le\ol{Q}(E)$ or
$\ul{Q}(E)\ge\ol{Q}(E)$ in general.
\end{corollary}

\begin{proof}
This is by Example \ref{ex:coin-tossing-finite} and Example \ref{ex:x_n=p_n}.
\end{proof}

Recall that, if Reality can strongly comply with a event $E$,
then $\overline{Q}(E)=1$.
Example \ref{ex:x_n=p_n} shows that the converse does not hold.


\section*{Acknowledgement}\label{sec:acknowlegement}

The authors appreciate the reviewer for the appropriate comments.
The first author was supported by Grant-in-Aid for JSPS fellows
and the second author by the Aihara Project,
the FIRST program from JSPS, initiated by CSTP.
Both authors were supported by
JSPS Grant-in-Aid for Challenging Exploratory Research Grant Number 25540009.


\end{document}